\documentclass[a4paper,11pt]{article}

\usepackage{amsmath,amssymb}
\usepackage{url}
\usepackage{graphicx}
\usepackage{epsfig}
\usepackage{color}

\graphicspath{{Fig/}}

\newcounter{exl}

\newenvironment{proof}{\begin{trivlist}
	\item[\noindent]{\it Proof\:}}{\quad $\square$\end{trivlist}}

\newenvironment{exl}{\begin{trivlist} \addtocounter{exl}{1}
	\item[\noindent]{\bf Example \theexl\:}}{\end{trivlist}}

\newenvironment{rem}{\begin{trivlist}
	\item[\noindent]{\bf Remark\:}}{\end{trivlist}}

\newtheorem{dfn}{Definition}[section]
\newtheorem{thm}{Theorem}
\newtheorem{prp}[dfn]{Proposition}
\newtheorem{lem}[dfn]{Lemma}
\newtheorem{cor}[dfn]{Corollary}

\def\R{{\mathbb R}}

\def\E{{\mathbb E}}
\def\H{{\mathbb H}}
\def\D{{\mathbb D}}
\def\Sph{{\mathbb S}}
\def\phi{\varphi}

\def\const{\mathrm{const}}
\def\span{\operatorname{span}}
\def\dist{\operatorname{dist}}
\def\im{\operatorname{im}}
\def\rk{\operatorname{rk}}
\def\pr{\operatorname{pr}}
\def\cR{\mathcal{R}}
\def\cF{\mathcal{F}}
\def\cQ{\mathcal{Q}}
\def\RP{{\mathbb R}{\mathrm{P}}}
\def\cV{\mathcal{V}}
\def\cE{\mathcal{E}}

\title{Projective background of the infinitesimal rigidity \newline of frameworks}
\author{Ivan Izmestiev\thanks{Research for this article was supported by the DFG Research Unit 565 ``Polyhedral Surfaces''.}}
\date{April 16, 2008}

\begin{document}

\maketitle

\begin{abstract}
We present proofs of two classical theorems. The first one, due to Darboux and Sauer, states that infinitesimal rigidity is a projective invariant; the other one establishes relations (infinitesimal Pogorelov maps) between the infinitesimal motions of a Euclidean framework and of its hyperbolic and spherical images.

The arguments use the static formulation of infinitesimal rigidity. The duality between statics and kinematics is established through the principles of virtual work. A geometric approach to statics, due essentially to Grassmann, makes both theorems straightforward. Besides, it provides a simple derivation of the formulas both for the Darboux-Sauer correspondence and for the infinitesimal Pogorelov maps.
\end{abstract}

\section{Introduction}
\label{sec:Intro}

\subsection{Infinitesimal rigidity}
A \emph{framework} is a collection of bars joined together at their ends by universal joints. A framework is called \emph{rigid}, if it cannot be flexed at the joints without deforming the bars; or, equivalently, if it can be moved only as a rigid body. The mathematical formalization of this is straightforward: a framework is a collection of points with distances between some pairs of them fixed; rigidity means that the points cannot be moved without changing one of those distances.

A framework is \emph{infinitesimally rigid} if its nodes cannot be moved in such a manner that the lengths of the bars remain constant in the first order. In practice, an infinitesimally flexible framework allows a certain amount of movement, even if it is rigid in the above sense.

A classical result on infinitesimal rigidity is the Legendre-Cauchy-Dehn theorem, \cite{Leg94}, \cite{Cau13}, \cite{Dehn16}: \emph{Every convex 3-dimensional polyhedron is infinitesimally rigid.} The theorem can be restated in the language of frameworks: The framework consisting of the vertices, edges and all face diagonals of a convex polyhedron is infinitesimally rigid. In fact, ``all face diagonals'' is redundant: it suffices to triangulate the faces arbitrarily, without adding new vertices. See \cite{Whi84a}, where this is generalized to higher dimensions.

For more information on different concepts of rigidity and an overview of results in this area, see the survey article \cite{Con93}.

Another classical but undeservedly little known result is the projective invariance of infinitesimal rigidity. For discrete structures it was first noticed by Rankine in 1863; it was proved by Darboux for smooth surfaces and by Liebmann and Sauer for frameworks, see Section \ref{subsec:rem}. Closely related to the projective invariance is the fact, discovered by Pogorelov, that a Euclidean framework can be turned into a hyperbolic or spherical one, respecting the infinitesimal rigidity.

The present paper contains proofs of these two properties of infinitesimal rigidity. The idea behind the proofs is not new, but we hope that a modern exposition might be useful. Our interest was stimulated by new applications that the projective properties of infinitesimal rigidity found in recent years, \cite{Scl05}, \cite{Scl06}, \cite{Fil08}.

There are further manifestations of the projective nature of the infinitesimal rigidity, such as its relations with polarity \cite{Whi87, Whi89} and Maxwell's theorem on projected polyhedra \cite{Whi82}.

Now let us state the two theorems in a precise way.

\subsection{Darboux-Sauer correspondence}
A framework $P$ in the Euclidean space $\E^d$ can be mapped by a projective map $\Phi:\RP^d \to \RP^d$ to another framework $\Phi(P)$. Here we assume that an affine embedding of $\E^d$ into $\RP^d$ is fixed and that $\Phi$ maps no vertex of $P$ to infinity. Frameworks $P$ and $\Phi(P)$ are called \emph{projectively equivalent}.
\begin{thm}[Darboux-Sauer correspondence]
\label{thm:DS}
Let $P$ and $P'$ be two projectively equivalent frameworks in $\E^d$. Then $P'$ is infinitesimally rigid iff $P$ is infinitesimally rigid. Moreover, the number of degrees of freedom of $P$ and $P'$ coincide.
\end{thm}
By the number of degrees of freedom of a framework we mean the dimension of the space of its infinitesimal motions modulo trivial ones. An infinitesimal motion is called trivial if it can be extended to an infinitesimal motion of $\E^d$ (equivalently, if it moves $P$ as a rigid body).

More specifically, let $\Phi$ be a projective map such that $P' = \Phi(P)$. Then $\Phi$ induces a bijection $\Phi^{\mathrm{kin}}$ between the space of infinitesimal motions of $P$ and the space of infinitesimal motions of $P'$ that maps trivial motions to trivial ones. We call the map $\Phi^{\mathrm{kin}}$ the (kinematic) \emph{Darboux-Sauer correspondence}.

\subsection{Infinitesimal Pogorelov maps}
Here is a simple way to describe the infinitesimal Pogorelov map. Consider a framework $P$ that is contained in a disk $\D^d \subset \E^d$. When the interior of $\D^d$ is viewed as Klein model of the hyperbolic space $\H^d$, the Euclidean framework $P$ turns into a hyperbolic framework $P^\H$. Pogorelov proved that $P^\H$ is infinitesimally rigid iff $P$ is; moreover, there is a canonical way to associate to every infinitesimal motion of $P$ an infinitesimal motion of $P^\H$ (with trivial motions going to trivial ones). This association is called the infinitesimal Pogorelov map.

Now let's be formal. Make the following identifications:
\begin{eqnarray}
\E^d & = & \{x \in \R^{d+1}|\; x^0 = 1\};\\
\H^d & = & \{x \in \R^{d+1}|\; x^0 > 0,\, \|x\|_{1,d}=1\}; \label{eqn:H^d}\\
\Sph^d & = & \{x \in \R^{d+1}|\; \|x\|=1\}, \label{eqn:S^d}
\end{eqnarray}
where $\|\cdot\|$ denotes the Euclidean norm, and $\|\cdot\|_{1,d}$ denotes the Minkowski norm of signature $(+,-,\ldots,-)$ in $\R^{d+1}$.

The projection from the origin of $\R^{d+1}$ defines the maps
\begin{eqnarray}
\Pi_\H: \D^d & \to & \H^d;\\
\Pi_\Sph: \E^d & \to & \Sph^d,
\end{eqnarray}
where $\D^d$ is the open unit disk in $\E^d \subset \R^{d+1}$ centered at $(1,0,\ldots,0)$.

To a framework $P$ in $\E^d$ there correspond frameworks $P^\H = \Pi_\H(P)$ and $P^\Sph = \Pi_\Sph(P)$ in $\H^d$ and $\Sph^d$. Note that $P^\H$ is defined iff $P \subset \D^d$.

\begin{thm}[Infinitesimal Pogorelov maps]
\label{thm:PogMaps}
Let $P$ be a framework in~$\E^d$. Then the following are equivalent:
\begin{itemize}
\item the Euclidean framework $P$ is infinitesimally rigid;
\item (for $P \subset \D^d$) the hyperbolic framework $P^\H$ is infinitesimally rigid;
\item the spherical framework $P^\Sph$ is infinitesimally rigid.
\end{itemize}
Moreover, frameworks $P$, $P^\H$, and $P^\Sph$ have the same number of degrees of freedom.
\end{thm}

Again, both statements of the theorem follow from the fact that there exist bijections between infinitesimal motions of frameworks $P$, $P^\H$, and $P^\Sph$ that map trivial motions to trivial ones. These bijections are called the \emph{infinitesimal Pogorelov maps}.

\subsection{Plan of the paper}
Section \ref{sec:Classical} contains preliminary material. The focus here is on the equivalence between infinitesimal rigidity and static rigidity expressed in Theorem \ref{thm:InfStat}. This theorem is a direct consequence of \emph{principles of virtual work} (Lemma~\ref{lem:Dual}).

Section \ref{sec:Proj} develops ``projective statics'' and ``projective kinematics''. The goal is to define motions and loads within projective geometry, which makes the projective invariance of infinitesimal rigidity straightforward. Geometric description of Darboux-Sauer correspondence is derived.

With infinitesimal rigidity defined in projective terms, it is not hard to relate the kinematics of frameworks $P$, $P^\H$, and $P^\Sph$. This is done in Section~\ref{sec:InfPog}, where formulas for the infinitesimal Pogorelov maps are also derived.

\subsection{Acknowledgements}
I am grateful to Walter Whiteley for inspiring discussions, and to Fran\c{c}ois Fillastre and Jean-Marc Schlenker for useful remarks.

\subsection{Examples}
\label{subsec:Exls}
Let us illustrate Theorem \ref{thm:DS} with some examples. Among the frameworks with a given combinatorics, the infinitesimally flexible ones sometimes have a nice geometric description. By Theorem \ref{thm:DS}, the description can always be made in projective terms.

\begin{exl}
Blaschke \cite{Bla20} and Liebmann \cite{Lie20} proved the following:
\begin{quote}
Let $P$ be a framework combinatorially equivalent to the skeleton of the octahedron. Color the triangles spanned by the edges of~$P$ black and white so that neighbors have different colors. The framework $P$ is infinitesimally flexible iff the planes of the four black triangles intersect, maybe at infinity. As a corollary, the planes of four white triangles intersect iff the planes of the four black ones do.
\end{quote}

\begin{figure}[ht]
\begin{center}
\input{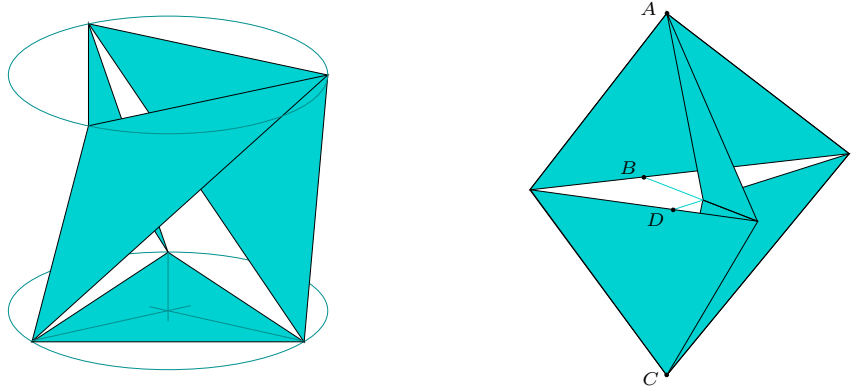}
\end{center}
\caption{Examples of infinitesimally flexible octahedra. Left: antiprism twisted by $90^\circ$. Right: the points $A$, $B$, $C$, $D$ lie in one plane.}
\label{fig:Blaschke}
\end{figure}

Figure \ref{fig:Blaschke} shows two configurations satisfying this condition. At the left is an example from \cite{Wun65}. It is obtained from a straight antiprism over a regular triangle by rotating one of the bases by $90^\circ$. It is easy to see that the horizontal shaded triangle is cut by the planes of the other three shaded triangles along its medians. Hence the four shaded planes intersect at a point. The right-hand example is due to Liebmann and is also depicted in \cite{Glu75}. Here the points $A$, $B$, $C$, and $D$ are assumed to lie in one plane. Since each of the four shaded planes contains one of the lines $AB$ or $CD$, they all pass through the intersection point of these lines.
\end{exl}

\begin{exl}
Consider the planar framework at the left of Figure \ref{fig:Desargues}. The lines matching the vertices of the two triangles are parallel. This implies that the velocity field represented by arrows is an infinitesimal motion. Hence, infinitesimally flexible will be any framework where the three matching lines are concurrent. In fact, this is a necessary and sufficient condition:
\begin{quote}
The planar framework on the right hand side of Figure \ref{fig:Desargues} is infinitesimally flexible iff the three lines $a$, $b$, $c$ intersect.
\end{quote}
Note that this condition is equivalent to the framework being a projection of a skeleton of a 3-polytope, so that the statement is a special case of Maxwell's theorem, \cite{Whi82}.
\end{exl}

\begin{figure}[ht]
\begin{center}
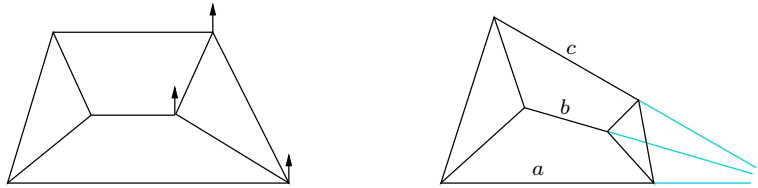
\end{center}
\caption{The framework at the left is infinitesimally flexible. The framework at the right is infinitesimally flexible iff the lines $a$, $b$, $c$ intersect.}
\label{fig:Desargues}
\end{figure}

\begin{exl}
Walter Whiteley \cite{Whi84c} shows how to derive from Theorem~\ref{thm:DS} the following statement:
\begin{quote}
Let $P$ be a framework in the Euclidean space $\E^d$ with combinatorics of a bipartite graph. If all of the vertices of $P$ lie on a non-degenerate quadric, then $P$ is infinitesimally flexible.
\end{quote}
Assume that all of the vertices of $P$ lie on the sphere. Move all the white vertices towards the center of the sphere, and all the black vertices in the opposite direction, see Figure \ref{fig:Bipartite}. It is easy to see that the distances between white and black vertices don't change in the first order. Thus $P$ is infinitesimally flexible. Since any non-degenerate quadric is a projective image of the sphere, $P$ is also infinitesimally flexible when it is inscribed in a quadric.

\begin{figure}[ht]
\begin{center}
\includegraphics{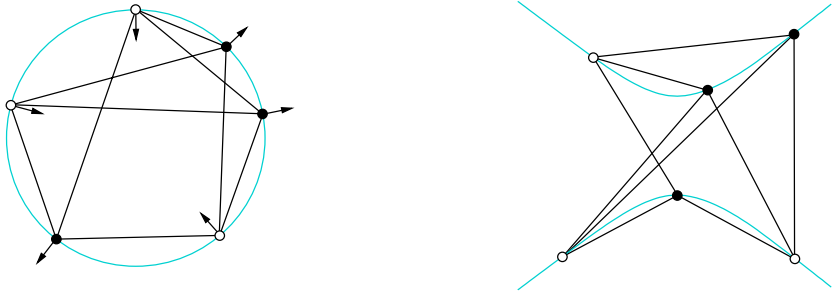}
\end{center}
\caption{The framework at the left is infinitesimally flexible. Due to the projective invariance of infinitesimal rigidity, the framework at the right is also infinitesimally flexible.}
\label{fig:Bipartite}
\end{figure}

The question about rigidity of bipartite frameworks was studied in \cite{BolRot80}. A characterization of infinitesimally flexible complete bipartite frameworks is given in \cite{Whi84b}.
\end{exl}

\section{Infinitesimal and static rigidity}
 \label{sec:Classical}
\subsection{Frameworks}
 \label{subsec:Frameworks}
Let $(\cV, \cE)$ be a graph with vertex set $\cV$ and edge set $\cE$. We denote the vertices by letters $i, j, \ldots$, and an edge joining the vertices $i$ and $j$ by $ij$.
\begin{dfn}
 \label{dfn:Framework}
A \emph{framework} in $\E^d$ with graph $(\cV, \cE)$ is a map
$$
\begin{array}{rrcl}
P: & \cV & \to & \E^d,\\
& i & \mapsto & p_i
\end{array}
$$
such that $p_i \ne p_j$ whenever $ij \in \cE$.
\end{dfn}
In other words, a framework is a straight-line drawing of a graph in $\E^d$, with self-intersections (even non-transverse ones) allowed. The motivation for studying frameworks comes from mechanical linkages; namely, the edges of a framework should be considered as rigid bars, and the vertices as universal joints.

Throughout the paper we assume that the vertices $(p_i)_{i \in \cV}$ of the framework span the space $\E^d$ affinely. This is not a crucial restriction: if the framework lies in an affine subspace of $\E^d$, then its infinitesimal motions can be decomposed into the direct sum of infinitesimal motions inside $\mathrm{span} \{p_i\}$ and arbitrary displacements in directions orthogonal to $\mathrm{span} \{p_i\}$.

\begin{rem}
We use different notations for the Euclidean space $\E^d$ and for the vector space $\R^d$. Informally speaking, $\E^d$ consists of points, $\R^d$ consists of vectors. We obtain $\E^d$ from $\R^d$ by ``forgetting'' the origin. The tangent space at every point of $\E^d$ is $\R^d$ with the standard scalar product. Also, every pair of points $p$, $p'$ in $\E^d$ determines a vector $p'-p \in \R^d$.
\end{rem}

\subsection{Infinitesimal rigidity}
 \label{subsec:ClassInf}
A continuous motion of the framework $P$ is a family $P(t)$ of frameworks ($t$ ranges over a neighborhood of $0$) such that $P(0) = P$ and the length of every bar does not depend on $t$:
\begin{equation}
 \label{eqn:BarLengths}
\|p_i(t) - p_j(t)\| = \const_{ij} \mbox{ for every }ij \in \cE.
\end{equation}
If $P(t)$ is differentiable, then differentiating \eqref{eqn:BarLengths} at $t=0$ yields
$$
\langle p_i - p_j, \dot p_i - \dot p_j \rangle = 0 \mbox{ for every }ij \in \cE.
$$

This motivates the following definition.

\begin{dfn}
 \label{dfn:InfMot}
A \emph{velocity field} on the framework $P$ is a map
$$
\begin{array}{rrcl}
Q: & \cV & \to & \R^d,\\
& i & \mapsto & q_i.
\end{array}
$$
A velocity field on $P$ is called an \emph{infinitesimal motion} of $P$ iff
\begin{equation}
 \label{eqn:InfMot}
\langle p_i - p_j, q_i - q_j \rangle = 0 \mbox{ for every }ij \in \cE.
\end{equation}
\end{dfn}

Since the conditions \eqref{eqn:InfMot} are linear in $Q$, infinitesimal motions of the framework $P$ form a vector space. Denote this vector space by $\cQ_{\mathrm{mot}}$.

Let $\{\Phi_t\}$ be a differentiable family of isometries of $\E^d$ such that $\Phi_0 = \mathrm{id}$. The vector field on $\E^d$ given by
$$
Q(x) = \left.\frac{d\Phi_t(x)}{dt}\right|_{t=0}
$$
is called an \emph{infinitesimal isometry} of $\E^d$. An infinitesimal motion of $P$ that is the restriction of an infinitesimal isometry of $\E^d$ is called \emph{trivial}. The space of trivial infinitesimal motions of $P$ is denoted by $\cQ_{\mathrm{triv}}$.

\begin{dfn}
\label{dfn:InfRig}
The framework $P$ is called \emph{infinitesimally rigid} iff all its infinitesimal motions are trivial.

The dimension of the quotient space $\cQ_{\mathrm{mot}}/\cQ_{\mathrm{triv}}$ is called the number of \emph{kinematic degrees of freedom} of the framework $P$.
\end{dfn}

The framework $P$ is called \emph{rigid} iff every continuous motion $P(t)$ has the form $\Phi_t \circ P$ with $\Phi_t$ a continuous family of isometries of $\R^d$. Intuition suggests that an infinitesimally rigid framework should be rigid. This is true \cite{Con80, RotWhi81}, but not straightforward since not every continuous motion can be reparametrized into a smooth one.

In the opposite direction, rigidity does not imply infinitesimal rigidity. Any of the frameworks on figures \ref{fig:Blaschke}--\ref{fig:Bipartite} can serve as an example.

\subsection{Static rigidity}
\label{subsec:ClassStat}
In the statics of rigid body, a force is represented as a line-bound vector. A collection of forces does not necessarily reduce to a single force.
\begin{dfn}
\label{dfn:Force}
A \emph{force} is a pair $(p,f)$ with $p \in \E^d$, $f \in \R^d$. A \emph{system of forces} is a formal sum $\sum_i (p_i,f_i)$ that may be transformed according to the following rules:
\begin{enumerate}
\setcounter{enumi}{-1}
\item a force with a zero vector is a zero force:
$$
(p,0) \sim 0;
$$
\item forces at the same point can be added and scaled as usual:
$$
\lambda_1(p,f_1) + \lambda_2(p,f_2) \sim (p, \lambda_1f_1+\lambda_2f_2);
$$
\item a force may be moved along its line of action:
$$
(p,f) \sim (p + \lambda f, f).
$$
\end{enumerate}
\end{dfn}

In $\E^2$, any system of forces is equivalent either to a single force or to a so called ``couple'' $(p_1, f) + (p_2, -f)$ with $p_1-p_2 \nparallel f$.

\begin{dfn}
\label{dfn:Load}
A \emph{load} on the framework $P$ is a map
$$
\begin{array}{rrcl}
F: & \cV & \to & \R^d,\\
& i & \mapsto & f_i.
\end{array}
$$
A load is called an \emph{equilibrium load} iff the system of forces $\sum_{i \in \cV} (p_i, f_i)$ is equivalent to a zero force.
\end{dfn}

A rigid body responds to an equilibrium load by interior stresses that cancel the forces of the load.

\begin{dfn}
\label{dfn:Stress}
A \emph{stress} on the framework $P$ is a map
$$
\begin{array}{rrcl}
\Omega: & \cE & \to & \R,\\
& ij & \mapsto & \omega_{ij}.
\end{array}
$$
The stress $\Omega$ is said to \emph{resolve} the load $F$ iff
\begin{equation}
\label{eqn:LoadResolv}
f_i + \sum_{j \in \cV} \omega_{ij} (p_j - p_i) = 0 \mbox{ for all }i \in \cV,
\end{equation}
where we assume $\omega_{ij} = 0$ for all $ij \notin \cE$.
\end{dfn}

We denote the vector space of equilibrium loads by $\cF_{\mathrm{eq}}$, and the vector space of resolvable loads by $\cF_{\mathrm{res}}$. It is easy to see that only an equilibrium load can be resolved: $\cF_{\mathrm{res}} \subset \cF_{\mathrm{eq}}$.

\begin{dfn}
The framework $P$ is called \emph{statically rigid} iff every equilibrium load on $P$ can be resolved.

The dimension of the quotient space $\cF_{\mathrm{eq}}/\cF_{\mathrm{res}}$ is called the number of \emph{static degrees of freedom} of the framework $P$.
\end{dfn}

\subsection{Relation between infinitesimal and static rigidity}
\label{subsec:InfStat}
Define a pairing between velocity fields and loads on the framework~$P$:
\begin{equation}
\label{eqn:Pairing}
\langle Q, F \rangle = \sum_{i \in \cV} \langle q_i, f_i \rangle.
\end{equation}
Clearly, this pairing is non-degenerate, thus it induces a duality between the space of velocity fields and the space of loads.

The following theorem provides a link between statics and kinematics of frameworks.

\begin{thm}
\label{thm:InfStat}
The pairing \eqref{eqn:Pairing} induces a duality
$$
\mathcal{Q}_{\mathrm{mot}}/\mathcal{Q}_{\mathrm{triv}} \cong \left( \mathcal{F}_{\mathrm{eq}}/\mathcal{F}_{\mathrm{res}} \right)^*
$$
between the space of non-trivial infinitesimal motions and the space of non-resolvable equilibrium loads.

In particular, a framework is infinitesimally rigid iff it is statically rigid.

For an infinitesimally flexible framework, the number of kinematic degrees of freedom is equal to the number of static degrees of freedom.
\end{thm}
\begin{proof}
This follows from Lemma \ref{lem:Dual} and from the canonical isomorphism $(V_1/V_2)^* \cong V_2^\perp/V_1^\perp$ for any pair of vector subspaces $V_1 \supset V_2$ of a space~$V$.
\end{proof}

\begin{lem}[Principles of virtual work]
\label{lem:Dual}
Under the pairing \eqref{eqn:Pairing},
\begin{enumerate}
\item
the space of infinitesimal motions is the orthogonal complement of the space of resolvable loads:
$$
\mathcal{Q}_{\mathrm{mot}} = (\mathcal{F}_{\mathrm{res}})^\perp;
$$
\item
the space of trivial infinitesimal motions is the orthogonal complement of the space of equilibrium loads:
$$
\mathcal{Q}_{\mathrm{triv}} = (\mathcal{F}_{\mathrm{eq}})^\perp.
$$
\end{enumerate}
\end{lem}
\begin{proof}
The space of resolvable loads is spanned by the loads $(F^{ij})_{ij \in \cE}$ with components
$$
\begin{array}{rcl}
f^{ij}_i & = & p_i - p_j\\
f^{ij}_j & = & p_j - p_i\\
f^{ij}_k & = & 0 \quad \mbox{ for }k \ne i, j.
\end{array}
$$
The orthogonality condition $\langle Q, F^{ij} \rangle = 0$ is equivalent to $\langle q_i-q_j, p_i - p_j \rangle = 0$. Thus $Q \in (\mathcal{F}_{\mathrm{res}})^\perp$ iff $Q$ is an infinitesimal motion, and the first principle is proved.

Let us prove that $\mathcal{Q}_{\mathrm{triv}} \supset (\mathcal{F}_{\mathrm{eq}})^\perp$. Let $Q$ be a velocity field that annihilates every equilibrium load. The load $F^{ij}$ defined in the previous paragraph is an equilibrium load for every $i,j \in \cV$ (with $ij$ not necessarily in $\cE$). The equations $\langle Q, F^{ij} \rangle = 0$ imply that $Q$ infinitesimally preserves pairwise distances between the points $(p_i)_{i \in \cV}$. Therefore $Q$ can be extended to an infinitesimal isometry of $\E^d$, that is $Q \in \mathcal{Q}_{\mathrm{triv}}$.

Let us prove $\mathcal{Q}_{\mathrm{triv}} \subset (\mathcal{F}_{\mathrm{eq}})^\perp$. Let $Q$ be the restriction of an infinitesimal isometry of $\E^d$. We have to show that $\langle Q, F \rangle = 0$ for every equilibrium load $F$. Since the system of forces $\sum_{i \in \cV} (p_i, f_i)$ corresponding to $F$ is equivalent to zero, there is a sequence of transformations as in Definition \ref{dfn:Force} that leads from $\sum_i (p_i, f_i)$ to $0$. It is not hard to show that the number $\langle Q, F \rangle$ remains unchanged after each transformation (if a force $(p',f')$ with a new application point $p'$ appears, then we substitute for $q'$ in the expression $\langle q', f' \rangle$ the velocity vector of our global infinitesimal isometry). Since $F$ vanishes at the end, we have $\langle Q, F \rangle = 0$ also at the beginning.
\end{proof}

\begin{cor}
$$
\dim \mathcal{F}_{\mathrm{eq}} = d\,|\cV| - \binom{d+1}2
$$
\end{cor}
\begin{proof}
Due to Lemma \ref{lem:Dual}, $\dim \mathcal{F}_{\mathrm{eq}} = d\,|\cV| - \dim \mathcal{Q}_{\mathrm{triv}} = d\,|\cV| - \binom{d+1}2$.
\end{proof}

Let $\Phi: \E^d \to \E^d$ be an affine isomorphism. The framework $P' = \Phi \circ P$ is called \emph{affinely equivalent} to $P$.

\begin{cor}
\label{cor:InfAff}
Infinitesimal rigidity is an affine invariant. Moreover, for any two affinely equivalent frameworks there is a canonical bijection between their infinitesimal motions that restricts to a bijection between trivial infinitesimal motions.

Explicitly, let $\Phi: x \mapsto Ax + b$ be an affine isomorphism of $\E^d$, written in an orthonormal coordinate system. Then the map that relates infinitesimal motions of $P$ with infinitesimal motions of $\Phi \circ P$ is $(A^*)^{-1}$.
\end{cor}
\begin{proof}
Static rigidity is affinely invariant in a straightforward way. Definitions in Section \ref{subsec:ClassStat} use only the affine structure of $\E^d$, and not the metric structure. Given an affine isomorphism $\Phi: x \mapsto Ax + b$, the transformation of forces $\Phi^{\mathrm{stat}}: f \mapsto Af$ maps equilibrium loads to equilibrium ones and resolvable to resolvable.

In order to obtain a transformation $\Phi^{\mathrm{kin}}$ of velocity fields, it suffices to require that $\langle \Phi^{\mathrm{kin}}(q), \Phi^{\mathrm{stat}}(f) \rangle = \langle q, f \rangle$ for any $q,f$. This implies the formula $(\Phi^{\mathrm{kin}})^{-1}: q \mapsto A^*q$.
\end{proof}

An alternative proof of the affine invariance of infinitesimal rigidity can be found in \cite{Bla20}.


\subsection{Rigidity matrix}
\label{subsec:RigMat}
The rigidity matrix is a standard tool for computing infinitesimal motions and the number of degrees of freedom of a framework.

\begin{dfn}
\label{dfn:RigMat}
The \emph{rigidity matrix} of a framework $P$ is an $\cE \times \cV$ matrix with vector entries:
$$
\cR = \; \scriptstyle{ij}
              \stackrel{\scriptstyle{i}}{\left(\begin{array}{ccc}
              & \vdots & \\
              \cdots & p_i - p_j & \cdots \\
              & \vdots &
             \end{array}\right)}.
$$
It has the pattern of the edge-vertex incidence matrix of the graph $(\cV, \cE)$, with $p_i - p_j$ on the intersection of the row $ij$ and the column $i$.
\end{dfn}

Note that the rows of $\cR$ are exactly the loads $(F^{ij})_{ij \in \cE}$ that span the space $\mathcal{F}_{\mathrm{res}}$, see the proof of Lemma \ref{lem:Dual}. The following proposition is just a reformulation of the first principle of virtual work (Lemma \ref{lem:Dual}, first part), together with its proof.

\begin{prp}
\label{prp:RigMat}
Consider $\cR$ as the matrix of a map $(\R^d)^\cV \to \R^\cE$. Then the following holds:
$$
\begin{array}{lcl}
\ker \cR & = & \mathcal{Q}_{\mathrm{mot}};\\
\im \cR^\top & = & \mathcal{F}_{\mathrm{res}}.
\end{array}
$$
\end{prp}

\begin{cor}
The framework is infinitesimally rigid iff
$$
\rk \cR = d\, |\cV| - \binom{d+1}2.
$$
\end{cor}
\begin{proof}
Indeed, $\rk \cR = d\, |\cV| - \dim \ker \cR$ which is equal to $d\, |\cV| - \dim \mathcal{Q}_{\mathrm{mot}}$ by Proposition \ref{prp:RigMat}. By definition, the framework is infinitesimally rigid iff $\mathcal{Q}_{\mathrm{mot}} = \mathcal{Q}_{\mathrm{triv}}$. Since $\mathcal{Q}_{\mathrm{mot}} \subset \mathcal{Q}_{\mathrm{triv}}$ and $\dim \mathcal{Q}_{\mathrm{triv}} = \binom{d+1}2$, the proposition follows.
\end{proof}

\section{Projective interpretation of rigidity}
\label{sec:Proj}
In this section we prove Theorem \ref{thm:DS}. For that, the static formulation of the infinitesimal rigidity turns out to be the most suitable. The proof amounts to redefining a force in projective terms, compatibly with Definition \ref{dfn:Force}. This is done in Section \ref{subsec:ProjStat}. In the same section we obtain formulas describing the correspondence between the loads in two projectively equivalent frameworks. In Section \ref{subsec:ProjKin} we derive from these formulas of static correspondence formulas of kinematic correspondence, using the duality from Section \ref{subsec:InfStat}.
Finally, we introduce projective analogs of notions of kinematics from Section \ref{subsec:ClassInf}.

Recall that we identify the Euclidean space $\E^d$ with the affine hyperplane $\{x^0 = 1\}$ of $\R^{d+1}$. This induces an affine embedding of $\E^d$ into $\RP^d$. We write points of $\RP^d$ as equivalence classes $[x]$ of points of $\R^{d+1} \setminus \{0\}$.

\subsection{Projective statics}
\label{subsec:ProjStat}
\begin{dfn}
\label{dfn:ProjFram}
A \emph{projective framework} with graph $(\cV,\cE)$ is a map
$$
\begin{array}{rrcl}
X: & \cV & \to & \RP^d,\\
& i    & \mapsto & [x_i]
\end{array}
$$
such that $[x_i] \ne [x_j]$ whenever $ij \in \cE$.
\end{dfn}

An affine embedding of $\E^d$ into $\RP^d$ associates a projective framework to every Euclidean framework.

\begin{dfn}
A force applied at a point $[x] \in \RP^d$ is a decomposable bivector divisible through $x$.
\end{dfn}
Thus every force at $[x]$ can be written as $x \wedge y \in \Lambda^2 \R^{d+1}$.

Let $(p,f)$ be a force in the sense of Definition \ref{dfn:Force}, i.e. $p \in \E^d \subset \R^{d+1}$, $f \in T_p \E^d \cong \R^d = \{x^0 = 0\}$. Associate with $(p,f)$ the bivector $p \wedge f$.

\begin{prp}
\label{prp:ProjStatCorr}
The extension of the map
\begin{equation}
\label{eqn:StatMap}
(p,f) \mapsto p \wedge f
\end{equation}
by linearity is well-defined and establishes an isomorphism of systems of forces on $\E^d$ with $\Lambda^2 \R^{d+1}$.
\end{prp}
\begin{proof}
The extension is well-defined since the equivalence relations from Definition \ref{dfn:Force} are respected by the map \eqref{eqn:StatMap}.

Let us prove that the map is surjective. It suffices to show that any decomposable bivector $x \wedge y \in \Lambda^2 \R^{d+1}$ is an image of a system of forces. If the plane spanned by $x$ and $y$ is not contained in $\R^d$, then there is a point $p \in \span\{x,y\} \cap \E^d$, and hence $x \wedge y = p \wedge f$ for an appropriate $f \in \R^d$. Otherwise, $x,y \in \R^d$. In this case represent $x$ as $x_1+x_2$ with $x_1,x_2 \notin \R^d$. Then the sum $x_1 \wedge y +x_2 \wedge y$ corresponds to a force couple.

To prove the injectivity, it suffices to show that the space of systems of forces on $\E^d$ has dimension at most $\binom{d+1}{2} = \dim \Lambda^2 \R^{d+1}$. This follows from an easy fact that any force can be written as a linear combination of forces from the set $\{(p_i, p_i - p_j)|\: i < j\}$, where $p_0, \ldots, p_d$ is a set of affinely independent points in $\E^d$.
\end{proof}

Due to Proposition \ref{prp:ProjStatCorr}, the following definitions are compatible with definitions of Section \ref{subsec:ClassStat}.

\begin{dfn}
A load on a projective framework $X$ is a map
$$
\begin{array}{rrcl}
G: & \cV & \to & \Lambda^2 \R^{d+1},\\
& i    & \mapsto & g_i,
\end{array}
$$
where $g_i$ is a force at $[x_i]$. A load is called an equilibrium load iff $\sum_{i \in \cV} g_i = 0$.
\end{dfn}

\begin{dfn}
Let $X$ be a projective framework with graph $(\cV, \cE)$. Denote by $\cE_{\mathrm{or}}$ the set of oriented edges: $\cE_{\mathrm{or}} = \{(i,j)|\: ij \in \cE\}.$ A stress on $X$ is a map
$$
\begin{array}{rrcl}
W: & \cE_{\mathrm{or}} & \to & \Lambda^2 \R^{d+1},\\
& (i,j)    & \mapsto & w_{ij}
\end{array}
$$
such that $w_{ij} \in \Lambda^2 \span\{x_i,x_j\}$ and $w_{ij} = -w_{ji}$.

The stress $W$ is said to resolve the load $G$ iff for all $i \in \cV$ we have
$$
g_i = \sum_j w_{ij}.
$$
\end{dfn}

\begin{prp}
\label{prp:DS_Stat}
Let $P$ and $P'$ be two frameworks in $\E^d \subset \RP^d$ such that $P' = \Phi \circ P$, where $\Phi: \RP^d \to \RP^d$ is a projective map. Then there is an isomorphism between the spaces of equilibrium loads on $P$ and $P'$ that maps resolvable loads to resolvable ones.
\end{prp}

\begin{proof}
Choose a representative $M \in GL(\R^{d+1})$ for~$\Phi$ and denote by $X$ and $X'$ the projective frameworks associated to $P$ and $P'$.

The map $M$ induces a map $M_*: \Lambda^2 \R^{d+1} \to \Lambda^2 \R^{d+1}$. Being a linear isomorphism, $M_*$ maps equilibrium loads on $X$ to equilibrium loads on $X'$, and resolvable ones to resolvable ones. Due to Proposition \ref{prp:ProjStatCorr}, the (projective) loads on $X$, respectively $X'$, nicely correspond to (Euclidean) loads on $P$, respectively $P'$. This yields the desired isomorphism between the spaces of loads on $P$ and $P'$.
\end{proof}

Denote by $\Phi^{\mathrm{stat}}$ the isomorphism between the spaces of loads on $P$ and $P'$ constructed in the proof of Proposition \ref{prp:DS_Stat}. It consists of a family of isomorphisms
$$
\Phi^{\mathrm{stat}}_p: T_p \E^d \to T_{\Phi(p)} \E^d, \quad p \in \E^d.
$$
Since the construction involves the choice of a representative $M$, the isomorphism $\Phi^{\mathrm{stat}}$ is determined only up to scaling. The next two propositions describe $\Phi^{\mathrm{stat}}$ explicitly.

\begin{prp}
\label{prp:A_p1}
Let $L \subset \E^d$ be the hyperplane that is sent to infinity by the projective map $\Phi$. Then
\begin{equation}
\label{eqn:A_p1}
\Phi^{\mathrm{stat}}_p(f) = h_L(p) h_L(p+f) \cdot \left(\Phi(p+f) - \Phi(p)\right), \quad \mbox{if } p+f \notin L,
\end{equation}
where $h_L$ denotes the signed distance to the hyperplane $L$.

In other words: to obtain $\Phi^{\mathrm{stat}}_p(f)$, map the application point and the endpoint of the vector $(p,f)$ by the map $\Phi$, and scale the resulting vector by the product of distances of these points from the hyperplane $L$.
\end{prp}
If $(p,f)$ is such that $p+f \in L$, then \eqref{eqn:A_p1} contains an indeterminacy. In this case the map $\Phi^{\mathrm{stat}}_p$ can be extended to $f$ by continuity or by linearity.

\begin{proof}
Consider $P$ as a projective framework. Choose a representative $M \in GL(\R^{d+1})$ of $\Phi$. The vector $\Phi^{\mathrm{stat}}_p(f)$ is uniquely determined by the equation
\begin{equation}
\label{eqn:CondA_p}
M_*(p \wedge f) = \Phi(p) \wedge \Phi^{\mathrm{stat}}_p(f)
\end{equation}
and the condition $\Phi^{\mathrm{stat}}_p(f) \in \R^d$.

Denote $x': = M(p)$. Then we have
$$
x' = \lambda \cdot \Phi(p).
$$
It is not hard to see that
$$
\lambda = c \cdot h_L(p)
$$
for some constant $c$ independent of $p$. Thus we have
\begin{eqnarray*}
M_*(p \wedge f) & = & M_*(p \wedge (p+f)) =  c^2 h_L(p) h_L(p+f) \cdot \Phi(p) \wedge \Phi(p+f)\\
& = & c^2 h_L(p) h_L(p+f) \cdot \Phi(p) \wedge \left(\Phi(p+f) - \Phi(p)\right).
\end{eqnarray*}
Choosing $M$ so that $c=1$, we obtain \eqref{eqn:A_p1} from \eqref{eqn:CondA_p}.
\end{proof}

\begin{prp}
\begin{equation}
\label{eqn:A_p2}
\Phi^{\mathrm{stat}}_p = h_L^2(p) \cdot \mathrm{d}\Phi_p,
\end{equation}
where $\mathrm{d}\Phi_p$ is the differential of the map $\Phi$ at the point $p$.
\end{prp}
\begin{proof}
This follows from \eqref{eqn:A_p1} by replacing $f$ with $tf$ and taking the limit as $t \to 0$.

There is also a simple direct proof. From the proof of Proposition \ref{prp:DS_Stat}, it is immediate that the vectors $\Phi^{\mathrm{stat}}_p(f)$ and $\mathrm{d}\Phi_p(f)$ are collinear for every~$f$. Since $\Phi^{\mathrm{stat}}_p$ and $\mathrm{d}\Phi_p$ are linear maps, this implies
\begin{equation}
\label{eqn:Prop}
\Phi^{\mathrm{stat}}_p = \lambda(p) \cdot \mathrm{d}\Phi_p,
\end{equation}
and it remains to determine the function $\lambda(p)$. Consider two arbitrary points $p_1$ and $p_2$ that are not mapped to infinity by $\Phi$ and the forces $p_2-p_1$ at $p_1$ and $p_1-p_2$ at $p_2$. Since these forces are in equilibrium, so must be their images. Thus we have
\begin{equation}
\label{eqn:condA_p2}
\frac{\lambda(p_1)}{\lambda(p_2)} = \frac{\|d\Phi_{p_2}(p_1-p_2)\|}{\|d\Phi_{p_1}(p_1-p_2)\|}.
\end{equation}
To compute the right hand side, restrict the map $\Phi$ to the line $p_1p_2$. In a coordinate system with the origin at the intersection point of $p_1p_2$ with the hyperplane $L$, this restriction takes the form $x \mapsto c/x$. Since the derivative of $c/x$ is proportional to $x^{-2}$, and $x$ is proportional to $h_L$, \eqref{eqn:Prop} and \eqref{eqn:condA_p2} imply \eqref{eqn:A_p2} (we can forget about $c$ because $\Phi^{\mathrm{stat}}$ is defined up to scaling).
\end{proof}

\subsection{Projective kinematics}
\label{subsec:ProjKin}
\begin{prp}
\label{prp:DS_Kin}
Let $P$ and $P'$ be two frameworks in $\E^d \subset \RP^d$ such that $P' = \Phi \circ P$, where $\Phi: \RP^d \to \RP^d$ is a projective map.  Then there is an isomorphism $\Phi^{\mathrm{kin}}$ between the infinitesimal motions of $P$ and $P'$ that maps trivial infinitesimal motions to trivial ones.

The map $\Phi^{\mathrm{kin}}$ consists of a family of isomorphisms $\Phi^{\mathrm{kin}}_p: T_p\E^d \to T_{\Phi(p)}\E^d$ given by
\begin{equation}
\label{eqn:Phi^kin}
\Phi^{\mathrm{kin}}_p = h_L^{-2}(p) \cdot (\mathrm{d}\Phi_p^{-1})^*,
\end{equation}
where $h_L$ denotes the signed distance to the hyperplane $L$ sent to infinity by~$\Phi$.
\end{prp}
\begin{proof}
This is a direct consequence of Theorem \ref{thm:InfStat}, Proposition \ref{prp:DS_Stat} and formula \eqref{eqn:A_p2}.
\end{proof}

For the sake of completeness and for the reason of curiosity, let us find the projective counterparts to the notions of kinematics.

Let $X$ be a framework in $\RP^d$ with graph $(\cV, \cE)$.
\begin{dfn}
A velocity vector at a point $[x] \in \RP^d$ is an element of the vector space $(\Lambda^2 \R^{d+1})^*/\Lambda^2 x^\perp$, where $x^\perp \subset (\R^{d+1})^*$ denotes the orthogonal complement of $x$.
\end{dfn}

Here is the motivation for this definition.

\begin{lem}
For a projective framework, the vector space of velocities at $[x]$ is dual to the vector space of forces at $[x]$.
\end{lem}
\begin{proof}
Indeed, the space of forces at $x$ is $x \wedge \R^{d+1} \subset \Lambda^2 \R^{d+1}$. For a subspace $W$ of a vector space $V$, there is a canonical isomorphism $W^* \cong V^*/W^\perp$. Since $(x \wedge \R^{d+1})^\perp = \Lambda^2 x^\perp$, the proposition follows.
\end{proof}

\begin{dfn}
A velocity field $(\tau_i)_{i \in \cV}$ on a projective framework $X$ is called an infinitesimal motion iff
\begin{equation}
\label{eqn:ProjInfMot}
\langle x_i \wedge x_j, \tau_i - \tau_j \rangle = 0 \mbox{ for every }ij \in \cE.
\end{equation}
An infinitesimal motion is called trivial iff there exists $\tau \in (\Lambda^2 \R^{d+1})^*$ such that $\tau_i = \tau + \Lambda^2 x_i^\perp$ for all $i \in \cV$.
\end{dfn}
Note that the difference $\tau_i - \tau_j \in (\Lambda^2 \R^{d+1})^*/(x_i \wedge x_j)^\perp$ is well-defined since $\Lambda^2 x_i^\perp \subset (x_i \wedge x_j)^\perp \supset \Lambda^2 x_j^\perp$.

Let us establish a correspondence with the notions of Section~\ref{subsec:ClassInf}. Recall that the Euclidean space $\E^d$ is identified with the hyperplane $\{x^0 = 1\} \subset \R^{d+1}$. Consider a framework $P$ in $\E^d$ as a projective framework $X$ with $x_i = p_i$. To any classical velocity vector $q \in T_p \E^d$ associate a projective velocity vector $\tau \in (p \wedge \R^{d+1})^*$ given by
\begin{equation}
\label{eqn:q_to_tau}
\langle p \wedge y, \tau  \rangle := \langle y, q \rangle \mbox{ for every }y \in T_p\E^d.
\end{equation}
(The angle brackets at the right hand side mean the scalar product in $T_p \E^d$.) Conversely, for every $\tau \in (\Lambda^2 \R^{d+1})^*/\Lambda^2 p^\perp$ consider the (well-defined) covector $p \lrcorner \tau \in p^\perp \subset (\R^{d+1})^*$. Restrict $p \lrcorner \tau$ to $\R^d$, identify $\R^d$ with $T_p \E^d$ by parallel translation, and identify $T_p \E^d$ with $T^*_p \E^d$ using the scalar product. Denote the result by~$q$:
\begin{equation}
\label{eqn:tau_to_q}
q := (p \lrcorner \tau|_{\R^d})^*.
\end{equation}
It is not hard to see that \eqref{eqn:q_to_tau} and \eqref{eqn:tau_to_q} define an isomorphism from $T_p \E^d$ to $(\Lambda^2 \R^{d+1})^*/\Lambda^2 x^\perp$ and its inverse.

\begin{lem} \label{lem:EucProj}
Let $P$ be a framework on $\E^d$, and let $X$ be the corresponding projective framework on $\RP^d$. Let $Q$ be a velocity field on $P$, and let $T$ be the velocity field on $X$ associated to $Q$ via \eqref{eqn:q_to_tau}. Then $T$ is an infinitesimal motion of $X$ if and only if $Q$ is an infinitesimal motion of $P$, and $T$ is trivial if and only if $Q$ is trivial.
\end{lem}
\begin{proof}
Equation \eqref{eqn:q_to_tau} implies
$$
\langle p_i - p_j, q_i - q_j \rangle = \langle p_i - p_j, q_i \rangle - \langle p_i - p_j, q_j \rangle = -\langle p_i \wedge p_j, \tau_i - \tau_j \rangle.
$$
Therefore $T$ satisfies \eqref{eqn:ProjInfMot} iff $Q$ satisfies \eqref{eqn:InfMot}.

Assume that $T$ is trivial: $\tau_i = \tau + \Lambda^2 x_i^\perp$ for some $\tau$. Then for every $i,j \in \cV$ we have
$$
\langle q_i - q_j, p_i - p_j \rangle = \langle p_i \lrcorner \tau - p_j \lrcorner \tau, p_i - p_j \rangle = 0,
$$
which implies that $Q$ can be extended to an infinitesimal isometry of $\E^d$. Thus if $T$ is trivial, so is $Q$. To prove the inverse implication, compute the dimensions of the spaces of trivial motions. For the Euclidean framework $P$ it is equal to $\binom{d+1}2$ (recall that $(p_i)_{i \in \cV}$ affinely span $\E^d$ by assumption). For the corresponding projective framework it is equal to the rank of the map
$$
(\Lambda^2 \R^{d+1})^* \to \bigoplus_{i \in \cV} (\Lambda^2 \R^{d+1})^*/\Lambda^2 p_i^\perp.
$$
Since the vectors $(p_i)_{i \in \cV}$ span $\R^{d+1}$, this map is injective, so its rank is equal to $\dim (\Lambda^2 \R^{d+1})^* = \binom{d+1}2$.
\end{proof}

\subsection{Remarks}
\label{subsec:rem}
Grassmann introduced in his book of 1844 ``Die lineale Ausdehn\-ungs\-lehre'' the bivector representation of forces acting on a rigid body (in terms of what we now call Grassmann coordinates). A good account on that is given in \cite{Kle09}. As Klein remarks, ``This book... is written in a style that is extraordinarily obscure, so that for decades it was not considered nor understood. Only when similar trains of thought came from other sources were they recognized belatedly in Grassmann's book.''

Once spelled out, the bivector representation of forces readily implies the projective invariance of static rigidity. Apparently, this was observed by Rankine in \cite{Ran63}, where he writes ``...theorems discovered by Mr. Sylvester ... obviously give at once the solution of the question''. Unfortunately, we don't know which theorems are meant; probably this is something similar to Proposition \ref{prp:ProjStatCorr}.

An exposition of these elegant but unfortunately little known ideas, along with additional references, can be found in \cite{CW82}, \cite{Whi84c}.

It seems that the observation of Rankine wasn't given much attention, because the next mention of the projective invariance of static rigidity I am aware of is 1920 in the paper \cite{Lie20} of Liebmann. Liebmann proves it only for frameworks with $|\cE| = d\, |\cV| - \binom{d+1}2$ that contain $d$ pairwise connected joints. In this case the rigidity matrix can be reduced to a square matrix by fixing the positions of these $d$ joints. Infinitesimal or static rigidity is then equivalent to vanishing of the determinant of this square matrix. Liebmann shows that the determinant is multiplied with a non-zero factor when the framework undergoes a projective transformation. This argument can probably be extended to the general case, but doesn't seem to produce a correspondence between the loads or velocity fields of two projectively equivalent frameworks.

Sauer in \cite{Sau35a} gives a proof of the projective invariance of static rigidity using Grassmann coordinates of forces and finds formula \eqref{eqn:A_p1}. In \cite{Sau35b}, Sauer proves the projective invariance of infinitesimal rigidity in an independent way.

For smooth surfaces in $\R^3$, the projective invariance of infinitesimal rigidity is proved by Darboux \cite{Dar96}.

Other proofs can be found in Wunderlich \cite{Wun82} for frameworks, and Volkov \cite{Vol74} for smooth manifolds.

The association $\Phi \mapsto \Phi^{\mathrm{stat}}$ as described by formulas \eqref{eqn:A_p1} and \eqref{eqn:A_p2} fails to be a functor. Namely, the equation
$$
(\Phi \circ \Psi)^{\mathrm{stat}}_p = \Phi^{\mathrm{stat}}_{\Psi(p)} \circ \Psi^{\mathrm{stat}}_p
$$
holds only up to a constant factor, because the definition of $\Phi^{\mathrm{stat}}$ involved choosing a representative $M \in GL(\R^{d+1})$ of $\Phi \in PGL(\R^{d+1})$. If one would like to have functoriality, one should choose the matrix $M$ in $SL_\pm(\R^{d+1}) = \{M \in GL(\R^{d+1})|\; \det M = \pm 1\}$. When $d$ is even, there are two possibilities, but they lead to the same map $\Phi^{\mathrm{stat}}$ due to $(-x) \wedge (-y) = x \wedge y$. Choosing $M$ in $SL_\pm(\R^{d+1})$ changes the formulas \eqref{eqn:A_p1} and \eqref{eqn:A_p2} by a constant factor that depends on $\Phi$. It would be interesting to know whether this factor has a geometric meaning.

Notions of statics clearly have a homological flavor: equilibrium loads are kind of cycles, resolvable loads are kind of boundaries. This is easy to formalize; in the projective interpretation of Section \ref{subsec:ProjStat} we have a chain complex
$$
\bigoplus_\cE \Lambda^2 \span\{x_i,x_j\} \stackrel{\delta}\longrightarrow \bigoplus_\cV x_i \wedge \R^{d+1} \stackrel{\epsilon}\longrightarrow \Lambda^2 \R^{d+1}
$$
with appropriately defined maps, so that $\ker \epsilon$ consists of equilibrium loads, and $\im \delta$ of resolvable loads on framework $X$. For kinematics, there is a dual cochain complex
$$
\bigoplus_\cE (\Lambda^2 \R^{d+1})^*/(x_i \wedge x_j)^\perp \stackrel{d}\longleftarrow \bigoplus_\cV (\Lambda^2 \R^{d+1})^*/\Lambda^2 x_i^\perp \stackrel{\iota}\longleftarrow (\Lambda^2 \R^{d+1})^*
$$
with $d = \delta^*, \iota = \epsilon^*$. The maps $\delta$ and $d$ can be expressed through the rigidity matrix defined in Section \ref{subsec:RigMat}.

There exist higher-dimensional generalizations of statics, see \cite{TWW95}, \cite{TW00}, \cite{Lee96}. By duality they are related to the algebra of weights \cite{McM93}, \cite{McM96}, and to the combinatorial intersection cohomology \cite{Bra06}. Algebraic properties of the arising chain complexes can be used to prove deep theorems on the combinatorics of simplicial polytopes \cite{Sta80}, \cite{Kal87}, \cite{McM93}.

\section{Infinitesimal Pogorelov maps}
\label{sec:InfPog}

\subsection{Proof of Theorem \ref{thm:PogMaps}}
The definitions of frameworks in $\H^d$ and $\Sph^d$ repeat Definition \ref{dfn:Framework}. Let us define infinitesimal motions.
\begin{dfn}
\label{dfn:HInfMot}
Let $P$ be a framework in $\H^d$ or $\Sph^d$ with graph $(\cV, \cE)$. A velocity field is a collection $(q_i)_{i \in \cV}$ of tangent vectors at the vertices of the framework: $q_i \in T_{p_i}\H^d$, respectively $q_i \in T_{p_i}\Sph^d$.

A velocity field $(q_i)$ is called an infinitesimal motion of $P$ iff
$$
\left.\frac{d}{dt}\right|_{t=0} \dist(p_i(t), p_j(t)) = 0
$$
for every family $(p_i(t))$ such that $p_i(0) = p_i$, $\dot p_i(0) = q_i$.

An infinitesimal motion is called trivial iff it is generated by a differentiable family of isometries of $\H^d$, respectively $\Sph^d$.
\end{dfn}

Recall that we identify $\H^d$ and $\Sph^d$ with subsets of $\R^{d+1}$ according to \eqref{eqn:H^d} and \eqref{eqn:S^d}. The following lemma is straightforward.
\begin{lem}
\label{lem:ScalProd}
A velocity field $(q_i)_{i \in \cV}$ is an infinitesimal motion of $P$ iff ${\langle p_i-p_j, q_i-q_j \rangle = 0}$ for every $ij \in \cE$. Here $\langle \cdot\, , \cdot \rangle$ denotes the Minkowski or the Euclidean scalar product in $\R^{d+1}$, according to whether $P$ is a hyperbolic or a spherical framework.
The tangent space at $p$ is identified with a vector subspace of $\R^{d+1}$.
\end{lem}

Embeddings \eqref{eqn:H^d} and \eqref{eqn:S^d} allow to associate with every framework $P$ in $\H^d$ or $\Sph^d$ a projective framework $X$. Exactly as in the Euclidean case, formulas \eqref{eqn:q_to_tau} and \eqref{eqn:tau_to_q} define a natural bijection between velocity fields on frameworks $P$ and $X$.

\begin{lem}
\label{lem:HypProj}
Let $P$ be a framework in $\H^d$ or $\Sph^d$, and let $X$ be the corresponding projective framework. Let $Q$ be a velocity field on $P$, and let $T$ be the velocity field on $X$ associated with $Q$. Then $T$ is an infinitesimal motion of $X$ if and only if $Q$ is an infinitesimal motion of $P$, and $T$ is trivial if and only if $Q$ is trivial.
\end{lem}
\begin{proof}
Due to Lemma \ref{lem:ScalProd}, the arguments from the proof of Lemma \ref{lem:EucProj} can be applied.
\end{proof}

Theorem \ref{thm:PogMaps} now follows from Lemmas \ref{lem:EucProj} and \ref{lem:HypProj}.

\subsection{Computing Pogorelov maps}
Let $P$ be a Euclidean framework with associated hyperbolic and spherical frameworks $P^\H$ and $P^\Sph$. Our proof of Theorem \ref{thm:PogMaps} shows that there are natural bijections between velocity fields on $P$, $P^\H$, and $P^\Sph$ that map infinitesimal motions to infinitesimal motions and respect the triviality property. Let us denote the vector fields associated with $Q = (q_i)$ by $Q^\H = (q^\H_i)$ and $Q^\Sph = (q^\Sph_i)$.
\begin{prp}
\label{prp:PogForm}
The velocity fields $Q$, $Q^\H$, and $Q^\Sph$ are related by the equations
\begin{eqnarray*}
q_i & = & \pr(\sqrt{1-\|p_i-c\|^2} \cdot q_i^\H);\\
q_i & = & \pr(\sqrt{1+\|p_i-c\|^2} \cdot q_i^\Sph),
\end{eqnarray*}
with $\|\cdot\|$ denoting the Euclidean scalar product. Here $c \in \E^d \subset \R^{d+1}$ is the point with coordinates $(1,0,\ldots,0)$ (the ``tangent point'' of $\E^d$, $\H^d$, and $\Sph^d$), and $\pr: \R^{d+1} \to \R^d$ is the projection $(x^0, x^1, \ldots, x^d) \mapsto (x^1,\ldots,x^d)$.
\end{prp}
\begin{proof}
For brevity, let us omit the index $i$ and denote by $p^\H$ the vertex of framework $P^\H$ corresponding to the vertex $p$ of $P$. Let $\tau$ be a velocity vector at the vertex $[p]$ in the underlying projective framework. By definition, we have
\begin{eqnarray*}
q & = & (p \lrcorner \tau|_{T_p\E^d})^*;\\
q^\H & = & (p^\H \lrcorner \tau|_{T_{p^\H}\H^d})^*.
\end{eqnarray*}
Since $p^\H \lrcorner \tau \in p^\perp$, we have $q^\H = (p^\H \lrcorner \tau)^*$, where this time $\alpha \mapsto \alpha^*$ denotes the isomorphism $\R^{d+1} \to (\R^{d+1})^*$ induced by the Minkowski scalar product. Also, it is not hard to show that $q = \pr((p \lrcorner \tau)^*)$. From
$$
p^\H = \sqrt{1-\|p_i-c\|^2} \cdot p
$$
we obtain the first formula of the proposition. The formula connecting $q$ with $q^\Sph$ is proved similarly, replacing the Minkowski scalar product in $\R^{d+1}$ with the Euclidean one.
\end{proof}

\subsection{Remarks}
A different derivation of the formulas of Proposition \ref{prp:PogForm} can be found in~\cite{SW07}.

In addition to infinitesimal Pogorelov maps there are finite Pogorelov maps, \cite{Pog73}. They associate with a pair of isometric hypersurfaces $P_1, P_2$ in $\E^d$ pairs $P_1^\H, P_2^\H$ and $P_1^\Sph, P_2^\Sph$ of isometric hypersurfaces in the hyperbolic and in the spherical space, respectively.

Liebmann in \cite{Lie20} proves the projective invariance of static rigidity for a certain class of frameworks, see Section \ref{subsec:rem}. After developing statics and kinematics in an arbitrary Cayley metric (which was also done by Lindemann \cite{Lin74}), he proves that the static rigidity of a framework doesn't depend on the choice of the metric.

In the smooth case, Volkov \cite{Vol74} proves that a map between Riemannian manifolds that sends geodesics to geodesics maps infinitesimally flexible hypersurfaces to infinitesimally flexible ones. Since projective maps of $\E^d$ to itself and gnomonic projections of the Euclidean space to the spherical and hyperbolic spaces send geodesics to geodesics, Volkov's theorem includes Darboux' and Pogorelov's as special cases.

There also exist infinitesimal Pogorelov maps to frameworks in the de Sitter space $\mathrm{d}\mathbb{S}^d$ (the one-sheeted hyperboloid $\{\|x\|_{1,d} = -1\}$ with the metric induced by the Minkowski metric). The metric on $\mathrm{d}\mathbb{S}^d$ is Lorentzian of constant curvature 1. The polar dual to a hyperbolic polyhedron is a de Sitter polyhedron. Thus, the Pogorelov map from $\H^d$ to $\mathrm{d}\mathbb{S}^d$ can be given a more geometric meaning, when the relations between polarity and infinitesimal rigidity are taken into account.

\bibliography{Rigidity}
\bibliographystyle{alpha}

\end{document}